\documentclass[12pt, a4paper]{amsart}

\usepackage{a4wide}
\usepackage[english]{babel}
\usepackage[T2A]{fontenc}
\usepackage[utf8]{inputenc} 
\usepackage{amsfonts}
\usepackage{amssymb, amsthm, amscd}
\usepackage{amsmath}
\usepackage{mathtools}
\usepackage{needspace}
\usepackage{etoolbox}
\usepackage{lipsum}
\usepackage{comment}
\usepackage{cmap}
\usepackage[pdftex]{graphicx}
\usepackage[unicode]{hyperref}
\usepackage[matrix,arrow,curve]{xy}
\usepackage[usenames,dvipsnames]{xcolor}
\usepackage{colortbl}
\usepackage{textcomp}
\usepackage{cite}
\usepackage{euscript}

\makeatletter
\def\@settitle{\begin{center}
		\baselineskip14\p@\relax
		\bfseries
		\LARGE
		\@title
	\end{center}
}
\makeatother

\newtheorem{theorem}{Theorem}[section]
\newtheorem{lemma}[theorem]{Lemma}
\newtheorem{proposition}[theorem]{Proposition}
\newtheorem{corollary}[theorem]{Corollary}
\theoremstyle{definition}
\newtheorem{definition}[theorem]{Definition}
\newtheorem*{remark}{Remark}
\newtheorem{example}[theorem]{Example}

\renewcommand{\leq}{\leqslant}
\newcommand{\R}{\mathbb{R}}
\newcommand{\Z}{\mathbb{Z}}
\newcommand{\A}{\mathbb{A}}
\newcommand{\CC}{\mathbb{C}}
\renewcommand{\G}{\mathbb{G}}
\newcommand{\GL}{\mathrm{GL}}
\newcommand{\Gal}{\mathrm{Gal}}
\newcommand{\Aut}{\mathrm{Aut}}
\newcommand{\Q}{\mathbb{Q}}

\newcommand{\Sch}{\mathrm{Sch}}
\newcommand{\Hom}{\mathrm{Hom}}
\newcommand{\Spec}{\mathrm{Spec}}

\title{Twisted forms of commutative monoid structures on affine spaces}
\author{Andrei V. Semenov and Pavel Gvozdevsky}

\thanks{Both authors are supported by ``Native Towns'', a social investment program of PJSC ``Gazprom Neft''. Also the first author is supported in part by The Euler International Mathematical Institute, grant number is 075-15-2019-16-20. The first author is a winner of Young Russian Mathematics award and he is grateful to its jury and sponsors.}

\address{Andrei V. Semenov:
Chebyshev Laboratory, 
St. Petersburg State University, 14th Line V.O., 29B, 
Saint Petersburg 199178 Russia}

\email{asemenov.spb.56@gmail.com}

\address{Pavel Gvozdevsky:
Chebyshev Laboratory, 
St. Petersburg State University, 14th Line V.O., 29B, 
Saint Petersburg 199178 Russia}

\email{gvozdevskiy96@gmail.com}


\begin{document}
\maketitle

\begin{abstract} In this paper, we study commutative algebraic monoid structures on affine spaces over an arbitrary field of characteristic zero. We obtain full classification of such structures on $\mathbb{A}_K^2$ and $\mathbb{A}_K^3$ and describe some generalizations on $\mathbb{A}_K^n$ for any dimension~ $n$. 
\end{abstract}
\vspace{10mm}

\section{Introduction}
In this paper, we study affine algebraic commutative monoids over an arbitrary field $K$ of characteristic zero. An affine algebraic monoid is an irreducible affine variety $M$ together with an assosiative multiplication map $\cdot : M \times M \longrightarrow M$, where map $\cdot$ is a morphism of algebraic varieties, and unit $e \in M(K)$ such that $em=me=m$ for all $m \in M$. A monoid $M$ is called commutative, if, in addition, $ab = ba$ for all $a, b \in M$. This article stands as a generalization of the paper \cite{1}, where affine algebraic commutative monoids were studied over algebraically closed field $L$ of characteristic zero, and classification of such monoids on $\mathbb{A}_L^2$ and $\mathbb{A}_L^3$ was reached. Our goal is to obtain such classification over an arbitrary field $K$ of characteristic zero, not necessary algebraically closed. \par 
It will be explained in Section 2.4 that all the monoids studied here are the twisted forms of monoids defined in \cite{1}, i.e. they are obtained by Galois descent. Basic facts about technique of Galois descent will be recalled in Section 2.3. Our main results, Theorem 4.1 and Theorem 5.3, claims that in dimensions 2 and 3 all nontrivial twisted forms can be described in terms of formulae that define multiplication and norm in separable algebras over the base field. \par 
In order to achieve these results we use the idea of the connection between a monoid $M$ and its group of invertible elements $G(M)$. Furthermore, we develop the connection between $\Aut(M)$ and $\Aut(G(M))$ and use the Galois descent technique applied to $\Aut(M)$.\par 
Nowadays the theory of affine algebraic monoids and group embeddings is a deeply developed area of mathematics lying at the intersection of algebra, algebraic geometry, combinatorics and representation theory. For further information and general presentations see \cite{7, 8, 10, 13}. Also let us note that same results on algebraic monoids over affine surfaces were obtained in \cite{2} and \cite{3}. \par 
 In Section 2 we recall some basic facts about schemes and algebraic monoids and develop our technique, including Galois descent and the first cohomology set of Galois group with coefficients in the authomorphism groups of monoids. \par 
 In Section 3 we obtain classification of commutative algebraic monoids on $\A_K^n$ of rank $0$, $n$ or $n-1$ (the notion of rank is introduced in Definition 2.3) for arbitrary field $K$ of characteristic zero. \par 
 In Sections 4 and 5 we reach a classification of all the monoid structures on $\mathbb{A}_K^2$ and $\mathbb{A}_K^3$ correspondently. \par 
 In the last section we give final remarks on automorphism groups of the obtained monoids, and formulate possible questions about further generalizations and applications of our technique.

Authors are grateful to professor Ivan Arzhantsev for helpful advises and productive discussions.

\section{Main definitions and facts}
Consider a field $K$ of characteristic zero and its algebraic closure $\overline{K}$. By $\Gal_{\overline{K}/K}$ we denote the Galois group of this extension. 

\subsection{Schemes as functors}

Throughout the paper we deal with schemes over $K$. A $K$-scheme $X$ is uniquely defined by its functor of points. This is a functor from the category of $K$-algebras to the category of sets, given by
$$
B\mapsto \Hom_{\Sch_K}(\Spec(B),X).
$$

There is an explicit description of functors that occurs this way, see \cite[Chapter VI]{4}. Thus we may identify these two notions. So if we want to define certain scheme, we just need to define the corresponding functor. In this approach, morphisms of schemes are the natural transformations of the corresponding functors. Affine schemes correspond to representable functors.

If $L/K$ is a field extension, then any $L$-scheme $X$ can be viewed as a $K$-scheme via functor $X_K(-)$ from category of $K$-algebras to the category of sets, where $X_K(B)$ consists of pairs $(q,x)$, such that $q\colon L\to B$ defines a structure of $L$-algebra on $B$, and $x \in X(B)$, where $B$ is viewed as an algebra over $L$ via $q$.

By \textit{variety} over $K$ we mean a reduced, separated $K$-scheme of finite type.

\subsection{Algebraic monoids}

\begin{definition}
An (affine) algebraic monoid over $K$ is an irreducible (affine) algebraic variety $M$ over $K$ with an associative multiplication $\cdot : M \times M \longrightarrow M$, which is a morphism of algebraic varieties, and an element $e\in M(K)$ which is a unit with respect to this multiplication.

If $M$ is isomorphic to an affine space $\A_K^n$ as a variety, we call it a monoid structure on~$\A_K^n$.

\end{definition}

\begin{lemma}
Let $M$ be a commutative algebraic monoid over $K$, whose underlying variety is geometrically irreducible (i.e. remains irreducible after the base change $K\to\overline{K}$). Then the functor $G$ from the category of $K$-algebras to the category of sets, given by
$$
G(B)=\{\text{invertible elements of }M(B)\},
$$
is an open subscheme of $M$.
\end{lemma}

\begin{proof}
This functor is naturally isomorphic to the functor
$$
G'(B)=\{(x,y)\in (M\times M)(B)\colon xy=e\},
$$
which is a fiber of a multiplication morphism $\cdot \colon M\times M\to M$. Therefore, $G$ is a scheme over $K$, and the embedding $i\colon G\to M$ is a morphism of schemes.

It follows from the paper \cite{9} that $i$ becomes an open immersion after the base change $K\to\overline{K}$, and it is known that if a morphism of schemes becomes an open immersion after the extension of the base field, then it is an open immersion. The last fact is a special case of \cite[Lemma 35.20.16]{12}.
\end{proof}

It is clear that $G(-)$ is a group scheme.

\begin{definition}
For a commutative monoid $M$ let us define the rank of $M$ (notation: $\mathrm{rank} (M)$) as the dimension of the maximal torus of the group of its invertible elements. Also by $\mathrm{corank} (M)$ we will denote the codimension of the maximal torus of the group of its invertible elements.
\end{definition}
It is easy to see that the group $G(M)$ of the invertible elements of $M$ is isomorphic to $\G_m^r \times \G_a^s$ under the field extension $\overline{K}$ (the proof of this well-known fact one may see in Chapter 2 of \cite{5}, for example), since $M$ is commutative (and so $G(M)$ is commutative too). Here $r = \mathrm{rank}(M)$ and $s = \mathrm{corank}(M)$, and we often will simply write $r$ and $s$ instead of $\mathrm{rank}(M)$ and $\mathrm{corank}(M)$.

\subsection{Galois descent}

Let $L/K$ be a finite Galois extension of fields with Galois group $\Gamma$. 
\begin{definition}
Let $Y$ be a scheme over $L$ on which $\Gamma$ acts by $K$-scheme automorphisms. Such action is called {\it compatible} if for any $\sigma\in \Gamma$ the diagram 
$$
\xymatrix{
Y\ar[r]^\sigma\ar[d] & Y\ar[d]\\
\mathrm{Spec}(L)\ar[r]^\sigma & \mathrm{Spec}(L)
}
$$
commutes.
\end{definition}

Сonsider the functor 
\begin{align*}
   F:  \{\text{quasiproj. } K \text{-schemes}\}& \longrightarrow \{\text{quasiproj. } L \text{-schemes with compatible  action of } \Gamma\}\\
    X&\mapsto X\otimes_K L,
\end{align*}
where $X\otimes L$ is the $L$-scheme defined by
$$
X\otimes L(B)=X(B).
$$
for any $L$-algebra $B$, and there is a Galois group action (via the canonical action on $L$) on $X\otimes_K L$ viewed as a $K$-scheme. Note that as a $K$-scheme $X\otimes L$ is just a product $X\times_{\Spec(K)} \Spec(L)$.

\begin{theorem}(Galois descent, see \cite[Section 4.4]{6}.)
The functor $F$ is an equivalence of categories.
\end{theorem} 

Therefore, this functor induces an equivalence between the categories of monoinds in the corresponding categories. Let us formulate it as follows.

\begin{corollary}
The functor

\begin{align*}
   F:  \{\text{quasiproj. } K \text{-monoids}\}& \longrightarrow {\left\{\begin{aligned}&\text{quasiproj. } L \text{-monoids with compatible  action of } \Gamma\\&\text{ preserving the monoidal structure}\end{aligned}\right\} }\\
   M&\mapsto M\otimes_K L,
\end{align*}

is an equivalence of categories.
\end{corollary}

Now let $M$ be an algebraic monoid over $K$. By $\Aut(M)$ we denote the functor
\begin{align*}
    \{\text{Field extensions }& L/K\}  \longrightarrow \{\text{Groups}\}\\
    L&\mapsto \{\text{Automorphisms of the } L\text{-scheme }M\otimes_K L\text{ preserving monoidal structure}\}.
\end{align*}

The action of $\Gamma$ on $L$ induces an action on $\Aut(M)(L)$ by functoriality: the scheme $\Aut(M)(L)$ acts on $M \otimes L$ by the polinomials with coefficients in $L$, since the coordinates on this scheme are taken over $K$. So one can define the action of $\Gamma$ on $\Aut(M)(L)$ as the action of $\Gamma$ on such coefficients. More precisely, 
 \[ \sigma(f) =\sigma \circ f\circ \sigma^{-1} \]
 where $\sigma$ on the right side is a map of $ M\otimes L$ (as a scheme over $K$).
 
\begin{definition}
A map $f\colon \Gamma\to \Aut(M)(L)$ is called a cocycle if it satisfies $f(\tau \cdot \rho)=f(\tau)\circ \tau \big( f(\rho) \big)$ for all $\rho, \tau \in \Gamma$, where $\cdot$ is the multiplication in $\Gamma$ and $\circ$ is the multiplication on $\Aut(M)(L)$.
\end{definition}
\begin{definition}
 Cocycles $f$ and $g$ are equivalent if there exists $c\in \Aut(M)(L)$, such that $g(\sigma)=c \circ f(\sigma) \circ [\sigma(c)]^{-1}$ for all $\sigma \in \Gamma$, where $\circ$ is the multiplication on $\Aut(M)(L)$. The set of equivalence classes will be denoted by $H^1(\Gamma,\Aut(M)(L))$.
\end{definition}
Obviously, the first cohomology is a pointed set, where the basepoint is the class of the trivial cocycle that maps everything to identity.
 \begin{definition}
An algebraic monoid $M'$ over $K$ is called a twisted $L/K$-form of $M$ if $M\otimes_K L$ and $M'\otimes_K L$ are isomorphic as algebraic monoids over $L$.
\end{definition}

It can be deduced from the facts above that there exists a bijection between isomorphism classes of $L/K$-forms of $M$
and $H^1(\Gamma, \Aut(M)(L))$, such that an $L/K$-form $M'$ with an isomorphism $\phi\colon M\otimes L\to M'\otimes L$ corresponds to the class of a cocycle $f$ defined by 
$$
f(\sigma)=\phi^{-1}\circ\sigma\circ \phi\circ \sigma^{-1} \in \Aut(M)(L),
$$
where $\circ$ denotes the composition of maps, cf. \cite[Remark 4.5.4]{6}.

We can recover $M'$ from a cocycle $f$ as follows
\begin{multline*}
M'(B)=\Hom_{\Sch_K}(\Spec(B),M')=\{x\in\Hom_{\Sch_L}(\Spec(B\otimes L),M'\otimes L)\mid \forall \sigma\in\Gamma\quad x\circ\sigma=\sigma\circ x \}=\\=\{x\in\Hom_{\Sch_L}(\Spec(B\otimes L),M\otimes L)\mid \forall \sigma\in\Gamma\quad \phi\circ x\circ\sigma=\sigma\circ \phi\circ x \}=\\
=\{x\in\Hom_{\Sch_L}(\Spec(B\otimes L),M\otimes L)\mid \forall \sigma\in\Gamma\quad f(\sigma)^{-1}\circ x=\sigma\circ x\circ\sigma^{-1} \}. \end{multline*}
Now one need to note that as a scheme the functor $M\otimes_K L (-)$ is equal to the functor $\Hom_{\Sch_L}(\Spec(-), M\otimes_K L)$ by definition. Also one can mention that $\sigma\circ x\circ\sigma^{-1} = \sigma(x)$ for any $\sigma \in \Gamma$ and any $x \in \Hom_{\Sch_L}(\Spec(B\otimes L),M\otimes L)$. So we have
\begin{multline*}
M'(B)=\{x\in\Hom_{\Sch_L}(\Spec(B\otimes L),M\otimes L)\mid \forall \sigma\in\Gamma\quad f(\sigma)^{-1}\circ x=\sigma\circ x\circ\sigma^{-1} \} = \\ 
=\{x\in M\otimes_K L(B\otimes_K L) \mid \forall \sigma\in\Gamma\quad f(\sigma)^{-1}(x)=\sigma(x)\},
\end{multline*}
where $B$ is any $K$-algebra and $\circ$ denotes the composition of maps as above.

The proofs for algebraic groups can be found, for example, in \cite{14}. This principle is very general, and the case of algebraic monoids is not that different. 

\begin{example}
Let $K=\R$, $L=\CC$, and $M=\G_m$. In this case $\Aut(M)(-)$ equals to the constant functor $\Z/2\Z$, i.e. for all field extensions $L/K$ we have $\Aut(\mathbb{G}_m)(L) = \Z/2\Z$.

Let $M'$ be the unit circle, i.e the group scheme over $\R$ defined by the equation $x_1^2+x_2^2=1$, where the multiplication is given by formulae that defines multiplication of complex numbers
$$
(x_1,x_2)*(y_1,y_2)=(x_1y_1-x_2y_2,x_1y_2+x_2y_1).
$$
There is an isomorphism 
$$
\phi\colon M\otimes \CC\to M'\otimes \CC,\quad z\mapsto (\tfrac{z+z^{-1}}{2},\tfrac{z-z^{-1}}{2i}).
$$
Thus $M'$ is a twisted form of $M$. Let $f$ be the corresponding cocycle. Let us calculate $f(\sigma)$ for $\sigma\in\Gal_{\CC/\R}$ be the complex conjugation. By definition we have $f(\sigma)=\phi^{-1}\circ\sigma\circ \phi\circ \sigma^{-1}$. We need to understand how $f(\sigma)$ acts on the functor $(M\otimes \CC)_{\R}(-)$, so fix an $\R$-algebra $B$, and fix $(q,z)\in (M\otimes \CC)_{\R}(B)$, i.e $q\colon \CC\to B$ gives us a structure of $\CC$-algebra and $z$ is an invertible element of $B$. Denote by $\overline{q}$ the composition of $q$ with the conjugation on $\CC$. We have
\begin{multline*}
   f(\sigma)(q,z)=(\phi^{-1}\circ\sigma\circ \phi\circ \sigma^{-1})(q,z)=(\phi^{-1}\circ\sigma\circ \phi)(\overline{q},z)=(\phi^{-1}\circ\sigma)(\overline{q},(\tfrac{z+z^{-1}}{2},\tfrac{z-z^{-1}}{2\overline{q}(i)}))=\\
   =(\phi^{-1}\circ\sigma)(\overline{q},(\tfrac{z+z^{-1}}{2},\tfrac{z^{-1}-z}{2q(i)}))=\phi^{-1}(q,(\tfrac{z+z^{-1}}{2},\tfrac{z^{-1}-z}{2q(i)}))=(q,z^{-1}).
\end{multline*}

We see that $f(\sigma)$ does not change the $\CC$-algebra structure, so it is an automorphism of $M\otimes \CC$ as a $\CC$-scheme, and there is only one such nontrivial isomorphism: $z\mapsto z^{-1}$. 

Now let us show how the inverse process works. Suppose that we have a cocycle $f$ that takes complex conjugation to the non-trivial automorphism of $M\otimes \CC$, and we want to recover the corresponding twisted form $M'$. We must use formula
$$
M'(B)=\{x\in M\otimes_K L(B\otimes_K L) \mid \forall \sigma\in\Gamma\quad f(\sigma)^{-1}(x)=\sigma(x)\},
$$
since $M'$ is the functor and we need to understand its structure. Now fix an $\R$-algebra $B$. We have
\[
M\otimes \CC(B\otimes \CC)=\{(x,y)\in (B\otimes \CC)^2\mid xy=1\}\simeq \]
\[\simeq \{(x_1,x_2,y_1.y_2)\in B^4\mid (x_1+ix_2)(y_1+iy_2)=1\}= \]
\[ = \{(x_1,x_2,y_1.y_2)\in B^4\mid x_1y_1-x_2y_2=1,\; x_1y_2+x_2y_1=0\}.
\]

In these terms complex conjugation acts by the rule $(x_1,x_2,y_1,y_2)\mapsto (x_1,-x_2,y_1,-y_2)$ and the non-trivial automorphism of $\G_m$ acts by the rule $(x_1,x_2,y_1,y_2)\mapsto (y_1,y_2,x_1,x_2)$. Thus we have
\begin{multline*}
M'(B)=\{(x_1,x_2,y_1.y_2)\in B^4\mid x_1y_1-x_2y_2=1,\; x_1y_2+x_2y_1=0,\; x_1=y_1, \; x_2=-y_2\}\simeq\\\simeq\{(x_1,x_2)\in B^2\mid x_1^2+x_2^2=1\}.
\end{multline*}
So $M'$ is the unit circle and the multiplication is given by formulae that defines multiplication of complex numbers.

\end{example}

\bigskip

Therefore, the isomorphism classes of all twisted forms of $M$ (union by all finite Galois extensions $L/K$) are classified by the set
\[
H^1(\Gal_{\overline{K}/K},\Aut(M))\overset{\mathrm{def}}{=}\varinjlim H^1(\Gal_{L/K},\Aut(M)(L)),
\]

where the limit is taken over all finite Galois extensions $K\le L\le \overline{K}$. 

\medskip

The following important statement will be referred as Hilbert's Theorem 90

\begin{theorem}(Hilbert's 90)
The sets $H^1(\Gal_{\overline{K}/K},\G_a)$ and $H^1(\Gal_{\overline{K}/K},\GL(n,-))$ are trivial for $n\ge 1$, i.e. they contain only the basepoint.
\end{theorem}

The proofs can be found, for example, in \cite{11}, Chapter 3, \S 1.1.

\subsection{Strategy of classification}

The paper \cite{1} classifies certain special cases of commutative monoid structures on affine spaces over a closed field. We intend to extend these results to an arbitrary field of characteristic zero. Note that all monoids described in \cite{1} are defined by formulae with integer coefficients. Therefore, we can consider monoids defined by the same formulae over our field $K$. Let us called such monoids split. 

Hence over $K$ all the monoids in question are the twisted forms of the split ones. We will classify them using Galois cohomology. Thus we need to study automorphisms of the split monoids. 

\begin{remark}
It is not obvious that all of those twisted forms are isomorphic to an affine space as a varieties, but that will follow a posteriori from our classification.
\end{remark}

\subsection{Automorphisms of algebraic monoids}

\begin{lemma}\label{Aut}
Let $M$ be a geometrically irreducible commutative algebraic monoid over $K$, and $G$ be its group of invertible elements. Then
\begin{enumerate}
    \item Any $g\in \Aut(G)(L)$ induces a birational automorphism of the scheme $M\otimes_K L$ (i.e. a scheme automorphism of a dense open subscheme) for any extension $L$ of $K$,
    
   \item The restriction map $\Aut(M)\to \Aut(G)$ is a monomorphism of functors and its image in $\Aut(G)(L)$ consists precisely of elements $g$ such that both birational maps induced by $g$ and $g^{-1}$ are regular (i.e. can be extended to a morphism on the whole $M\otimes_K L$).
\end{enumerate}
\end{lemma}
\begin{proof}
Follows easily from the fact that $G$ is a dense open subscheme of $M$.
\end{proof}

\begin{lemma}\label{AutGm}
Over any field we have $\Aut(\G_m^n) \simeq \GL(n, \Z)$.
\end{lemma}
\begin{proof}
Consider an authomorphism $\alpha \in \Aut(\G_m^n)$. So
\[ \alpha (1  \dots \underset{\substack{\uparrow \\ \text{i-th place}}} x  \dots 1) = (x^{\alpha_{1i}} \ldots x^{\alpha_{ii}} \dots  x^{\alpha_{ni}}) \text{ for all } 1 \leq i \leq n.
\]
Now let us construct map $f: \Aut(\G_m^n) \longrightarrow \GL(n, \Z)$ by the rule $f(\alpha) = (\alpha_{i,j})_{i,j=1}^n$: clearly $(\alpha_{i,j})_{i,j=1}^n \in \GL(n, \Z)$ since $\alpha$ is bijective. Now fix $A = (\beta_{i,j})_{i,j=1}^n \in \GL(n, \Z)$ and construct $\beta \in \Aut(\G_m^n)$ by the rule
\[ \beta(x_1, \dots , x_n) = (\prod \limits_i x_i^{\beta_{1i}}, \dots , \prod \limits_i x_i^{\beta_{ni}}), \]

Obviously $\beta \in \Aut(\G_m^n)$, so we obtain $g : \GL(n, \Z) \longrightarrow \Aut(\G_m^n)$ by putting $g(A) = \beta$. Now it remains to observe that $fg = id$, $gf = id$, and that $f $ is a morphism of algebraic groups: all these facts could be checked by direct computations.
\end{proof}

\section{Three particular cases}

In this section, we consider the cases of rank 0, corank 0 and corank 1.

\begin{proposition}\label{rank0}
Any commutative monoid on $\A^n$ of rank 0 is isomorphic to $\G_a^n$, i.e.
$$(x_1,\ldots,x_n)*(y_1,\ldots,y_n)=(x_1+y_1,\ldots,x_n+y_n).$$
\end{proposition}

\begin{proof}
From Proposition 1 of \cite{1} it follows that such monoid M becomes isomorphic to $\G_a^n$ after a field extension. Automorphisms group of $\G_a^n$ is $\GL(n,-)$. By Hilbert's Theorem 90, its first Galois cohomology set is trivial. Therefore, $\G_a^n$ has no nontrivial twisted forms. Hence any such monoid is isomorphic to $\G_a^n$ already over the base field.  
\end{proof}

\begin{proposition}\label{corank0}
Any commutative monoid on $\A^n$ of corank 0 is isomorphic to the multiplicative monoid of a separable algebra over $K$ (i.e. a direct product of field extensions). Moreover, non-isomorphic algebras give rise to non-isomorphic monoids. 
\end{proposition}

\begin{proof}
It follows from \cite[Proposition 1]{1} that any such monoid becomes isomorphic to $$(x_1,\ldots,x_n)*(y_1,\ldots,y_n)=(x_1y_1,\ldots,x_ny_n).$$
after a field extension.

Let us denote by $M$ the monoid defined by that last formula. Its group of invertible elements is $\G_m^n$. By Lemma 2.12 $\Aut (\G_m^n)$ is a constant group scheme $\GL(n, \Z)$. A matrix from $\GL(n, \Z)$ induces a regular map on $M$ if and only if all its entries are non-negative (since it should be a regular map on {\it monoids}, not groups). Therefore, by Lemma \ref{Aut} we obtain that the automorphisms of $M$ are matrices $g\in \GL(n, \Z)$ such that both $g$ and $g^{-1}$ have all the entries non-negative. It is easy to see that these are precisely the permutation matrices. Thus $\Aut(M)=S_n$ and it acts on $M$ by permutations of coordinates.

Note that $M$ is a multiplicative monoid of the algebra $K^n$. Any twisted form of $M$ is obtained by Galois descent with cocycle in $H^1(\mathrm{\Gal_{L/K}},S_n)$ for some finite Galois extension $L/K$. Since any permutation of coordinates is in fact an automorphism of the algebra $L^n$, this cocycle $f$ actually gives a twisted form $A_f$ of the algebra $K^n$ and our monoid is a multiplicative monoid of $A_f$. Finally one can note that any twisted form of the algebra $K^n$ is a direct product of field extensions, so $A_f$ is a separable algebra. Hence the first statement holds.

The second statement follows from the fact that both algebras and monoids are classified by $H^1(\mathrm{\Gal_{\overline{K}/K}},S_n)$.
\end{proof}

\begin{example}
Let $K=\R$ and $n=2$. Then any monoid on $\A^2_{\R}$ of corank 0 is isomorphic either to componentwise multiplication
$$(x_1,x_2)*(y_1,y_2)=(x_1y_1,x_2y_2),$$
or to multiplication of complex numbers
$$(x_1,x_2)*(y_1,y_2)=(x_1y_1-x_2y_2,x_1y_2+x_2y_1).$$
\end{example}

\begin{example}
Let $K=\Q$ and $n=2$. Then any monoid on $\A^2_{\Q}$ of corank 0 is isomorphic either to componentwise multiplication or to multiplication in $\Q(\sqrt{d})$
$$(x_1,x_2)*(y_1,y_2)=(x_1y_1+dx_2y_2,x_1y_2+x_2y_1),$$
where $d \in \mathbb{Q}$ is not a square. Monoids for $d_1$ and $d_2$ are isomorphic iff the fields $\Q(\sqrt{d_1})$ and $\Q(\sqrt{d_2})$ are isomorphic, i.e iff $d_1/d_2$ is a square.
\end{example}

\begin{proposition}\label{corank1}
For any commutative monoid on $\A^n$ of corank 1 there exists a collection of separable algebras $R_1$,$\ldots$,$R_k$ of dimensions $d_1$,$\ldots$, $d_k$ and integers $0\le c_1<\ldots<c_k$ such that $d_1+\ldots+d_k=n-1$ and the monoid structure is given by
\begin{multline}
(\overline{x}_1,\ldots,\overline{x}_k,x)*(\overline{y}_1,\ldots,\overline{y}_k,y)=(\overline{x}_1\overline{y}_1,\ldots,\overline{x}_k\overline{y}_k,N(\overline{x}_1)^{c_1}\ldots N(\overline{x}_k)^{c_k}y+\\+N(\overline{y}_1)^{c_1}\ldots N(\overline{y}_k)^{c_k}x),\label{Norms}
\end{multline}

\noindent  where $\overline{x_i}$, $\overline{y_i}$ are $d_i$-tuples of coordinates, $\overline{x}_i\overline{y}_i$ is calculated by formula that defines multiplication in $R_i$ in some fixed basis, and $N(\overline{x}_i)$, $N(\overline{y}_i)$ are calculated by formula that computes norm in $R_i$. Moreover, numbers $d_i$, $c_i$ and isomorphism classes of $R_i$ are defined uniquely by the isomorphism class of the monoid. 
\end{proposition}

\begin{proof}
It follows from \cite[Proposition 1]{1}, that after a field extension such monoid becomes isomorphic to $$
(x_1,\ldots,x_n)*(y_1,\ldots,y_n)=(x_1y_1,\ldots,x_{n-1}y_{n-1},x_1^{b_1}\ldots x_{n-1}^{b_{n-1}}y_n+y_1^{b_1}\ldots y_{n-1}^{b_{n-1}}x_n)
$$
for some uniquely defined integers $0\le b_1\le\ldots\le b_{n-1}$.

Let us denote by $M$ the monoid defined by the last formula. Its group of invertible elements is $G=\G_m^{n-1}\times \G_a$. Since there are no nontrivial homomorphisms between $\G_a$ and $\G_m$, we have $\Aut(G)=\GL(n-1,\Z)\times \G_m$, where $\GL(n-1,\Z)$ is a constant group scheme.

It is easy to see that in order for the birational morphism of $M$ defined by $g\in \GL({n-1},\Z)\times \G_m$ to be regular, it is necessary that its first component has all entries non-negative. Matrices $g\in \GL(n-1, \Z)$, such that both $g$ and $g^{-1}$ have only non-negative entries, are permutation matrices. So by Lemma \ref{Aut} we obtain that the image of the restriction map $\Aut(M)\to \Aut(G)$ is contained in $S_{n-1}\times \G_m$.

Fix $(\sigma,z)\in S_{n-1}\times\G_m$. Let us calculate the corresponding birational map. The embedding $i\colon \G_m^{n-1}\times \G_a \to M=\A^n$ is given by 
$$
i(t_1,\ldots,t_{n-1},v)= (t_1,\ldots,t_{n-1},t_1^{b_1}\ldots t_{n-1}^{b_{n-1}}v).
$$

Thus such map is $i\circ (\sigma,z)\circ (i^{-1})$. It is easy to see that
$$
i^{-1}(x_1,\ldots,x_{n-1},x_n)=(x_1,\ldots,x_{n-1},x_1^{-b_1}\ldots x_{n-1}^{-b_{n-1}}x_n),
$$
so
$$
(\sigma,z)\cdot (x_1,\ldots,x_{n-1}x_1^{-b_1}\ldots x_{n-1}^{-b_{n-1}}x_n)=(x_{\sigma^{-1}(1)},\ldots,x_{\sigma^{-1}(n-1)},zx_1^{-b_1}\ldots x_{n-1}^{-b_{n-1}}x_n),
$$
finally
$$
i(x_{\sigma^{-1}(1)},\ldots,x_{\sigma^{-1}(n-1)},zx_1^{-b_1}\ldots x_{n-1}^{-b_{n-1}}x_n)=\left(x_{\sigma^{-1}(1)},\ldots,x_{\sigma^{-1}(n-1)},z\prod x_i^{b_{\sigma^{-1}(i)}-b_i}x_n\right).
$$

For this map to be regular it is necessary and sufficient that $b_{\sigma^{-1}(i)}\ge b_i$ for all $i$. Since the permutation has finite order, this means that $b_{\sigma^{-1}(i)}= b_i$ for all $i$, or equivalently $b_{\sigma(i)}= b_i$ for all $i$.

Since $b_1\le\ldots\le b_{n-1}$, we may put
\[ b_1=\ldots=b_{d_1}=c_1, \quad b_{d_1+1}=\ldots=b_{d_1+d_2}=c_2,\]
$$\ldots$$
\[b_{d_1+\ldots+d_{k-1}+1}=\ldots=b_{d_1+\ldots+d_k}=c_k,\]
where $\sum \limits_{i=1}^k d_i =n-1$ and $c_1<\ldots<c_k$. Note that $c_i$ and $d_i$ are defined uniquely by $b_i$.

Grouping together the variables that correspond to equal $b_i$, we can rewrite the multiplication in the monoid $M$ as follows
\begin{multline*}
(\overline{x}_1,\ldots,\overline{x}_k,x)*(\overline{y}_1,\ldots,\overline{y}_k,y)=(\overline{x}_1\overline{y}_1,\ldots,\overline{x}_k\overline{y}_k,N(\overline{x}_1)^{c_1}\ldots N(\overline{x}_k)^{c_k}y+\\+N(\overline{y}_1)^{c_1}\ldots N(\overline{y}_k)^{c_k}x),
\end{multline*}
where $\overline{x_i}$, $\overline{y_i}$ are $d_i$-tuples of coordinates, $\overline{x}_i \overline{y}_i$ is the componentwise multiplication, and $N(\overline{x}_i)$ is the product of all entries of the $d_i$-th tuple $\overline{x}_i$ (i.e. the norm in the algebra $L^n$).

It follows from Lemma \ref{Aut} and facts proved above that $\Aut(M)=\prod S_{d_i}\times \G_m$, where permutations act by permuting coordinates inside $d_i$-tuples, and $\G_m$ acts by scaling the last coordinate.

When $d_i$ and $c_i$ are fixed, the twisted forms of $M$ are classified by the set $H^1(\Gal_{\overline{K}/K},\prod S_{d_i}\times \G_m)$. On the other hand, the set $H^1(\Gal_{\overline{K}/K},\prod S_{d_i})$ classifies collections of separable algebras $R_1$,$\ldots$,$R_k$ of dimensions $d_1$,$\ldots$,$d_k$. The embedding $\prod S_{d_i} \hookrightarrow \prod S_{d_i}\times \G_m$ induces the map
$$
f\colon H^1(\Gal_{\overline{K}/K},\prod S_{d_i}) \longrightarrow H^1(\Gal_{\overline{K}/K},\prod S_{d_i}\times \G_m).
$$

To any collection of algebras $R_1$,$\ldots$,$R_k$ of dimensions $d_1$,$\ldots$,$d_k$ this  map assigns a twisted form $M'$ of $M$. It is easy to see that this assignment gives us the monoid with the multiplication given by the formula (\ref{Norms}) via the collection  $R_1$,$\ldots$,$R_k$, so this map is defined by the formula (\ref{Norms}). It remains to notice that cohomology, as functor of coefficients, preserve products; and  $H^1(\Gal_{\overline{K}/K},\G_m)$ is trivial by Hilbert's Theorem 90, so the map $f$ is bijective.
\end{proof}

\begin{example}
Let $K=\R$, $n=3$, $k=1$, $d_1=2$, and $R_1=\CC$. Then the corresponding monoid structure on $\A_{\R}^3$ is given by
\[
(x_1,x_2,x_3)*(y_1,y_2,y_3)=(x_1y_1-x_2y_2,x_1y_2+x_2y_1,(x_1^2+x_2^2)^cy_3+(y_1^2+y_2^2)^cx_3)
\]
for some $c \in \mathbb{Z}_{\ge 0}$.
\end{example}

\begin{example}
Let $K=\Q$, $n=3$, $k=1$, $d_1=2$, and $R_1=\Q(\sqrt{d})$ for $d \in \mathbb{Q}$. Then the corresponding monoid structure on $\A_{\Q}^3$ is given by
\[
(x_1,x_2,x_3)*(y_1,y_2,y_3)=(x_1y_1+dx_2y_2,x_1y_2+x_2y_1,(x_1^2-dx_2^2)^cy_3+(y_1^2-dy_2^2)^cx_3)
\]
for some $c \in \mathbb{Z}_{\ge 0}$.
\end{example}

\section{Main results for $\A_K^1$ and $\A_K^2$}
Since all commutative monoid structures on $\A_K^1$ and $\A_K^2$ are covered by Propositions \ref{rank0}, \ref{corank0}, and \ref{corank1} we obtain the following classification result.

\begin{theorem}
1)Every commutative monoid on $\A_K^1$ is isomorphic to one of the following monoids:
\begin{gather*}
    A\colon\quad (x)*(y)=(x+y);\\
    M\colon\quad (x)*(y)=(xy).
\end{gather*}

2)Every commutative monoid on $\A_K^2$ is isomorphic to one of the following monoids:

\medskip

\begin{tabular}{c|c|cc}
    rank & Notation  &  $(x_1, x_2)*(y_1, y_2)$ \\
    \hline
       \hline
    0    & $2A$        &  $(x_1+y_1,x_2+y_2)$\\
       \hline
    1    & $M\underset{b}{+}A$       &  $(x_1y_1,x_1^by_2+y_1^bx_2)$, & $b\in \Z_{\ge 0}$\\
       \hline
    2    & $M+M$       &  $(x_1y_1,x_2y_2)$\\
       \hline
    2    & $M(L)$      & Multiplicative monoid of a quadratic extension $L/K$.
\end{tabular}

\medskip

Moreover these monoids are pairwise non-isomorphic.
\end{theorem}

\begin{corollary}
Assume that $K=\R$. Then every commutative monoid on $\A_{\R}^1$ or $\A_{\R}^2$ is isomorphic either to one of the monoids described in \cite[Proposition 2]{1}, or to one in Example 3.3.
\end{corollary}

\begin{corollary}
Assume that $K=\Q$. Then every commutative monoid on $\A_{\Q}^1$ or $\A_{\Q}^2$ is isomorphic either to one of the monoids described in \cite[Proposition 2]{1}, or to one in Example 3.4.
\end{corollary}

\section{Main results for $\A_K^3$}
Now we consider the case $n=3$.
\begin{proposition}\label{bc}
Let $b,c\in \Z$, $0 \le b\le c$. Then the commutative monoid $M$ on $\A_K^3$ given by
$$
(x_1,x_2,x_3)*(y_1,y_2,y_3)=(x_1y_1,x_1^by_2+y_1^bx_2,x_1^cy_3+y_1^cx_3)
$$
has no nontrivial twisted forms.
\end{proposition}
\begin{proof}
Let us compute the group $\Aut(M)$. The group of invertible elements of this monoid is equal to $G=\G_m\times \G_a^2$. Since there are no nontrivial homomorphisms between $\G_a$ and $\G_m$, we have $\Aut(G)=\Z/2\Z\times \GL(2, -)$, where $\Z/2\Z$ is a constant group scheme.

It is easy to see that in order to obtain a {\it regular} morphism on $M$, defined by $g\in \Aut(M)$, the $\Z/2\Z$-component of $g$ must be trivial. Thus by Lemma \ref{Aut} the image of the restriction map $\Aut(M)\to \Aut(G)$ is contained in $\GL(2,-)$.

Take a matrix $\left(\begin{smallmatrix}\alpha & \beta \\ \gamma &\delta \end{smallmatrix}\right)\in \GL(2,L)$ for some extension $L/K$. Let us calculate the corresponding birational map. The embedding $i\colon \G_m\times \G_a^2 \to M=\A^3$ is given by
$$
i(t,v_1,v_2)=(t,t^bv_1,t^cv_2).
$$
Thus such map is $i\circ \left(\begin{smallmatrix}\alpha & \beta \\ \gamma &\delta \end{smallmatrix}\right)\circ (i^{-1})$. It is easy to see that
$$
i^{-1}(x_1,x_2,x_3)=(x_1,x_1^{-b}x_2,x_1^{-c}x_3),
$$
so
$$
\left(\begin{smallmatrix}\alpha & \beta \\ \gamma &\delta \end{smallmatrix}\right)\cdot (x_1,x_1^{-b}x_2,x_1^{-c}x_3)=(x_1,\alpha x_1^{-b}x_2+\beta x_1^{-c}x_3,\gamma x_1^{-b}x_2+\delta x_1^{-c}x_3),
$$
and
$$
i(x_1,\alpha x_1^{-b}x_2+\beta x_1^{-c}x_3,\gamma x_1^{-b}x_2+\delta x_1^{-c}x_3)=(x_1,\alpha x_2+\beta x_1^{b-c}x_3,\gamma x_1^{c-b}x_2+\delta x_3).
$$

There are two cases.

{\bf Case} $b=c$. In this case, every matrix from $\GL(2,L)$ induces a regular map on $M$. Therefore, $\Aut(M)=\GL(2,-)$, and by Hilbert's 90 Theorem $H^1(\Gal_{\overline{K}/K},\Aut(M))$ is trivial.

{\bf Case} $b<c$. In this case, a matrix $\left(\begin{smallmatrix}\alpha & \beta \\ \gamma &\delta \end{smallmatrix}\right)$ from $\GL(2,L)$ induces a regular map on $M$ iff $\beta=0$. Therefore, $\Aut(M)$ is a group of lower-triangular invertible 2 by 2 matrices, which is isomorphic to a semi-direct product $\G_a\leftthreetimes \G_m^2$. Applying long exact sequence in cohomology (see \cite{11}, Chapter 1, \S 5.5, Proposition 38)
$$
\ldots\to H^1(\Gal_{\overline{K}/K},\G_a)\to H^1(\Gal_{\overline{K}/K},\G_a\leftthreetimes \G_m^2)\to H^1(\Gal_{\overline{K}/K},\G_m^2)\to \ldots,
$$
and Hilbert's Theorem 90 we again obtain that the set $H^1(\Gal_{\overline{K}/K},\Aut(M))$ is trivial.
\end{proof}

\begin{proposition}\label{bcQ}
Let $b,c\in\Z_{> 0}$, $b\le c$. Define $d$ and $e$ by $c=bd+e$, $d,e\in\Z$, $0\le e<b$. Denote by $Q_{b,c}$ the polynomial
$$
Q_{b,c}(x_1,y_1,x_2,y_2)=\frac{(x_1^by_2+y_1^bx_2)^{d+1}-(x_1^by_2)^{d+1}-(y_1^bx_2)^{d+1}}{x_1^{b-e}y_1^{b-e}}.
$$
Then the commutative monoid $M$ on $\A_K^3$ given by
$$
(x_1,x_2,x_3)*(y_1,y_2,y_3)=(x_1y_1,x_1^by_2+y_1^bx_2,x_1^cy_3+y_1^cx_3+Q_{b,c}(x_1,y_1,x_2,y_2))
$$
has no nontrivial twisted forms.
\end{proposition}
\begin{proof}
Let us compute the group $\Aut(M)$. As before, the group of invertible elements is equal to $G=\G_m\times \G_a^2$. Then $\Aut(G)=\Z/2\Z\times \GL(2,-)$, and it is easy to see that the image of the restriction map $\Aut(M)\to \Aut(G)$ is contained in $\GL(2,-)$.

Take a matrix $\left(\begin{smallmatrix}\alpha & \beta \\ \gamma &\delta \end{smallmatrix}\right)\in \GL(2,L)$ for some extension $L/K$. Let us calculate the corresponding birational map. The embedding $i\colon \G_m\times \G_a^2 \to M=\A^3$ is given by
$$
i(t,v_1,v_2)=(t,t^bv_1,t^c(v_2+v_1^{d+1})).
$$
Thus such map is $i\circ \left(\begin{smallmatrix}\alpha & \beta \\ \gamma &\delta \end{smallmatrix}\right)\circ (i^{-1})$. It is easy to see that
$$
i^{-1}(x_1,x_2,x_3)=(x_1,x_1^{-b}x_2,x_1^{-c}x_3-(x_1^{-b}x_2)^{d+1}),
$$
so
\begin{multline*}
\left(\begin{smallmatrix}\alpha & \beta \\ \gamma &\delta \end{smallmatrix}\right)\cdot (x_1,x_1^{-b}x_2,x_1^{-c}x_3-(x_1^{-b}x_2)^{d+1})=\\=(x_1,\alpha x_1^{-b}x_2+\beta x_1^{-c}x_3-\beta(x_1^{-b}x_2)^{d+1},\gamma x_1^{-b}x_2+\delta x_1^{-c}x_3-\delta(x_1^{-b}x_2)^{d+1}),
\end{multline*}
finally
\begin{multline*}
i(x_1,\alpha x_1^{-b}x_2+\beta x_1^{-c}x_3-\beta(x_1^{-b}x_2)^{d+1},\gamma x_1^{-b}x_2+\delta x_1^{-c}x_3-\delta(x_1^{-b}x_2)^{d+1})=\\=(x_1,\alpha x_2+\beta x_1^{b-c}x_3-\beta x_1^{-bd}x_2^{d+1}, \star),
\end{multline*}
where we have
\[ \star = x_1^c\bigg(\gamma x_1^{-b}x_2+\delta x_1^{-c}x_3-\delta(x_1^{-b}x_2)^{d+1} + \big(\alpha x_1^{-b}x_2+\beta x_1^{-c}x_3-\beta(x_1^{-b}x_2)^{d+1}\big)^{d+1}\bigg). \]
Now we see that for the map in question to be regular, it is necessary that $\beta=0$, because of the summand $\beta x_1^{-bd}x_2^{d+1}$. Now we finish our computation, assuming that $\beta=0$:
$$
i(x_1,\alpha x_1^{-b}x_2,\gamma x_1^{-b}x_2+\delta x_1^{-c}x_3-\delta(x_1^{-b}x_2)^{d+1})=(x_1,\alpha x_2,\gamma x_1^{c-b}x_2+\delta x_3+(\alpha^{d+1}-\delta)x_1^{e-b}x_2^{d+1}).
$$
We now see that this map is regular iff $\delta=\alpha^{d+1}$. Therefore,
$$
\Aut(M)(L)=\left\{\left(\begin{smallmatrix}\alpha & 0 \\ \gamma &\alpha^{d+1} \end{smallmatrix}\right)\in \GL(2,L) \right\}.
$$
This group is isomorphic to a semi-direct product $\G_a\leftthreetimes \G_m$. Applying long exact sequence in cohomology
$$
\ldots\to H^1(\Gal_{\overline{K}/K},\G_a)\to H^1(\Gal_{\overline{K}/K},\G_a\leftthreetimes \G_m)\to H^1(\Gal_{\overline{K}/K},\G_m)\to \ldots,
$$
and Hilbert's 90, we again obtain that the set $H^1(\Gal_{\overline{K}/K},\Aut(M))$ is trivial.
\end{proof}

For a monoid $M$ on $\A_K^3$ we have $rank(M) \in \{0,1,2,3\}$. The cases $r=0,3,2$ are covered by Propositions \ref{rank0}, \ref{corank0}, and \ref{corank1} correspondingly. It remains to consider the case $r=2$. 

It follows from \cite[Theorem 1]{1} that any monoid on $\A_K^3$ of rank 2 after a field extension becomes isomorphic either to the one considered in Proposition \ref{bc} or to the one considered in Proposition \ref{bcQ}. Applying these Propositions, we obtain that such isomorphism exists already over the base field $K$.

To sum up we obtain the following classification result.

\begin{theorem}
Every commutative monoid on $\A_K^3$ is isomorphic to one of the following monoids:

\medskip

\!\!\!\!\!\!\!\!\!\!\!\!\!\!\!\!\!\!\!\!\!\begin{tabular}{c|c|cc}
    rank & Notation  &  $(x_1, x_2,x_3)*(y_1, y_2,y_3)$ \\
    \hline
       \hline
    0    & $3A$        &  $(x_1+y_1,x_2+y_2,x_3+y_3)$\\
       \hline
    1    & $M\underset{b}{+}A\underset{c}{+}A$       &  $(x_1y_1,x_1^by_2+y_1^bx_2,x_1^cy_3+y_1^cx_3)$, & $b,c\in \Z_{\ge 0}$, $b\le c$,\\
       \hline
    1    & $M\underset{b}{+}A\underset{b,c}{+}A$       &  $(x_1y_1,x_1^by_2+y_1^bx_2,x_1^cy_3+y_1^cx_3+Q_{b,c}(x_1,y_1,x_2,y_2))$, & $b,c\in \Z_{> 0}$, $b\le c$,\\
       \hline
    2    & $M+M\underset{b,c}{+}A$       &  $(x_1y_1,x_2y_2,x_1^bx_2^cy_3+y_1^by_2^cx_3)$, & $b,c\in \Z_{\ge 0}$, $b\le c$,\\
       \hline
    2    & $M^2(L)\underset{c}{+}A$      & $((x_1,x_2)*_{M(L)}(y_1,y_2), N_L(x_1,x_2)^cy_3+N_L(y_1,y_2)^cx_3)$ & $c\in\Z_{\ge 0}$, $\deg L/K =2$ ,\\
       \hline
    3    & $M+M+M$      & $(x_1y_1,x_2y_2,x_3y_3)$, \\
       \hline
    3    & $M^2(L)+M$     & $((x_1,x_2)*_{M(L)}(y_1,y_2),x_3y_3)$, & \\
       \hline
    3    & $M^3(L)$       & Multiplicative monoid of a cubic extension $L/K$. 
\end{tabular}

\medskip

Here the polynomial $Q_{b,c}$ is defined in Proposition \ref{bcQ}. By $(x_1,y_1)*_{M(L)}(x_2,y_2)$ we mean the formula that defines multiplication in $L$ in some fixed basis, and by $N_L(x_1,x_2)$ we mean the formula that defines norm in $L$ in the same basis.   

Moreover these monoids are pairwise non-isomorphic.
\end{theorem}

\begin{corollary}
Assume that $K=\R$. Then every commutative monoid on $\A_{\R}^3$ is isomorphic to one of the following monoids

\medskip

\!\!\!\!\!\!\!\!\!\!\!\!\!\!\!\!\!\!\!\!\!\begin{tabular}{c|c|cc}
    rank & Notation  &  $(x_1, x_2,x_3)*(y_1, y_2,y_3)$ \\
    \hline
       \hline
    0    & $3A$        &  $(x_1+y_1,x_2+y_2,x_3+y_3)$\\
       \hline
    1    & $M\underset{b}{+}A\underset{c}{+}A$       &  $(x_1y_1,x_1^by_2+y_1^bx_2,x_1^cy_3+y_1^cx_3)$, & $b,c\in \Z_{\ge 0}$, $b\le c$,\\
       \hline
    1    & $M\underset{b}{+}A\underset{b,c}{+}A$       &  $(x_1y_1,x_1^by_2+y_1^bx_2,x_1^cy_3+y_1^cx_3+Q_{b,c}(x_1,y_1,x_2,y_2))$, & $b,c\in \Z_{> 0}$, $b\le c$,\\
       \hline
    2    & $M+M\underset{b,c}{+}A$       &  $(x_1y_1,x_2y_2,x_1^bx_2^cy_3+y_1^by_2^cx_3)$, & $b,c\in \Z_{\ge 0}$, $b\le c$,\\
       \hline
    2    & $M^2(\CC)\underset{c}{+}A$      & $(x_1y_1-x_2y_2,x_1y_2+x_2y_1,(x_1^2+x_2^2)^cy_3+(y_1^2+y_2^2)^cx_3)$ & $c\in\Z_{\ge 0}$, \\
       \hline
    3    & $M+M+M$      & $(x_1y_1,x_2y_2,x_3y_3)$, \\
       \hline
    3    & $M^2(\CC)+M$     & $(x_1y_1-x_2y_2,x_1y_2+x_2y_1,x_3y_3)$. & 
\end{tabular}

\end{corollary}

\medskip

\begin{corollary}
Assume that $K=\Q$. Then every commutative monoid on $\A_{\Q}^3$ is isomorphic either to one of the following monoids

\medskip

\!\!\!\!\!\!\!\!\!\!\!\!\!\!\!\begin{tabular}{c|c|cc}
    rank & Notation  &  $(x_1, x_2,x_3)*(y_1, y_2,y_3)$ \\
    \hline
        \hline
    0    & $3A$        &  $(x_1+y_1,x_2+y_2,x_3+y_3)$\\
        \hline
    1    & $M\underset{b}{+}A\underset{c}{+}A$       & 
    $(x_1y_1,x_1^by_2+y_1^bx_2,x_1^cy_3+y_1^cx_3)$, & $b,c\in \Z_{\ge 0}$, $b\le c$,\\
        \hline
    1    & $M\underset{b}{+}A\underset{b,c}{+}A$       &  $(x_1y_1,x_1^by_2+y_1^bx_2,x_1^cy_3+y_1^cx_3+Q_{b,c}(x_1,y_1,x_2,y_2))$, & $b,c\in \Z_{> 0}$, $b\le c$,\\
        \hline
    2    & $M+M\underset{b,c}{+}A$       &  $(x_1y_1,x_2y_2,x_1^bx_2^cy_3+y_1^by_2^cx_3)$, & $b,c\in \Z_{\ge 0}$, $b\le c$,\\
        \hline
        
    2    & $M^2(\Q(\sqrt{d}))\underset{c}{+}A$      & $\big(x_1y_1+dx_2y_2,x_1y_2+x_2y_1$, &  $c\in\Z_{\ge 0}$ and\\  
     & & $(x_1^2-dx_2^2)^cy_3+(y_1^2-dy_2^2)^cx_3\big)$ & $d\in\Q$ isn't a square\\
     
         \hline
    3    & $M+M+M$      & $(x_1y_1,x_2y_2,x_3y_3)$, \\
        \hline
    3    & $M^2(\Q(\sqrt{d}))+M$     & $(x_1y_1+dx_2y_2,x_1y_2+x_2y_1,x_3y_3)$, & $d\in\Q$ isn't a square \\
        \hline
    3    & $M^3(L)$       & Multiplicative monoid of a cubic extension $L/\Q$. 

\end{tabular}

\end{corollary}

\medskip

\begin{example}
Let $K=\Q$, and $L=\Q(\sqrt[3]{2})$. Then the multiplication in $M^3(L)$ is defined by
$$(x_1,x_2,x_3)*(y_1,y_2,y_3)=(x_1y_1+2x_2y_3+2x_3y_2, \ x_1y_2+x_2y_1+2x_3y_3, \ x_1y_3+x_2y_2+x_3y_1).$$
\end{example}

\begin{example}
Let $K=\Q$, and $L=\Q[\alpha]/(\alpha^3+\alpha+1)$. Then the multiplication in $M^3(L)$ is defined by
\[(x_1,x_2,x_3)*(y_1,y_2,y_3)= \]
\[=(x_1y_1-x_2y_3-x_3y_2, x_1y_2+x_2y_1-x_2y_3-x_3y_2-x_3y_3,x_1y_3+x_2y_2+x_3y_1-x_3y_3).\]
\end{example}

\section{Final remarks}
Note that the proofs of our results contain the calculation of the automorphism group of the described monoids. Indeed calculation of the automorphism group for the split monoids are performed in the proofs directly; and the automorphisms of an arbitrary monoid $M$ are the automorphisms of $M\otimes \overline{K}$ that commutes with the action of the Galois group. We list the resulting automorphism groups here.

\medskip

\begin{tabular}{c|c}
  $(x_1, \ldots,x_n)*(y_1, \ldots,y_n)$ & $\Aut(M)$ \\
    \hline
        \hline
    $(x_1+y_1,\ldots,x_n+y_n)$ & $\GL(n,K)$,\\
        \hline
    Multiplicative monoid of a separeble algebra $R$ & $\Aut(R)$,\\
        \hline
    Formula from Proposition \ref{corank1} & $\Aut(R_1)\times\ldots\times\Aut(R_k)\times \G_m$,\\    
        \hline
    $(x_1y_1,x_1^by_2+y_1^bx_2,x_1^by_3+y_1^bx_3)$ & $\GL(2,K)$,\\
        \hline
    $(x_1y_1,x_1^by_2+y_1^bx_2,x_1^cy_3+y_1^cx_3)$, $b>c$ & $\left\{\left(\begin{smallmatrix}\alpha & 0 \\ \gamma &\beta \end{smallmatrix}\right)\in \GL(2,K) \right\}$,\\
        \hline
    $(x_1y_1,x_1^by_2+y_1^bx_2,x_1^cy_3+y_1^cx_3+Q_{b,c}(x_1,y_1,x_2,y_2))$ & $\left\{\left(\begin{smallmatrix}\alpha & 0 \\ \gamma &\alpha^{d+1} \end{smallmatrix}\right)\in \GL(2,K) \right\}$.

\end{tabular}

\medskip

Now we give some suggestions on possible generalizations of our results. In \cite{2, 3} monoids on affine surfaces were studied over an algebraically closed field of characteristic zero. We believe that our technique can help to generalise the results of these papers to the case of non-closed fields of characteristic zero. So we formulate\par
{\bf Question 1}: is it possible to reach similar classification for monoids on affine surfaces over an arbitrary (not necessary algebraically closed) field $K$ of characteristic zero? \par 

However, obtaining such classification in positive characteristic can be much harder problem even if the ground field is closed, because of the nontrivial commutative unipotent groups, see, for example, \cite[Exercise 8.8]{14}. So let us formulate \par 
{ \bf Question 2}: is it possible to reach similar classification for monoids on affine spaces $\A_K^n$ over a field $K$ of characteristic $p$, where $p$ is a prime number? Consider first the cases where $K$ is algebraically closed and $n = 2$.\par

\end{document}